\documentclass[a4paper]{article}
\usepackage[]{geometry}
\usepackage[english]{babel}
\usepackage[utf8x]{inputenc}
\usepackage{amsmath,amssymb,amsthm,amsfonts,amssymb,scrextend,color}
\usepackage{graphicx}
\usepackage[colorinlistoftodos]{todonotes}
\usepackage{array}
\usepackage{tabu}
\newtheorem*{definition*}{Definition}
\newtheorem*{notation*}{Notation}
\newtheorem{theorem}{Theorem}[subsection]
\newtheorem*{theorem*}{Theorem}
\newtheorem*{claim*}{Claim}
\newtheorem*{remark*}{Remark}

\numberwithin{equation}{section}

\title{Evaluation of Harmonic Sums with Integrals}
\author{Vivek Kaushik and Daniele Ritelli}

\begin{document}

\maketitle

\begin{abstract}
We consider the sums $S(k)=\sum_{n=0}^{\infty}\frac{(-1)^{nk}}{(2n+1)^k}$ and $\zeta(2k)=\sum_{n=1}^{\infty}\frac{1}{n^{2k}}$  with $k$ being a positive integer. We evaluate these sums with multiple integration, a modern technique. First, we start with three different double integrals that have been previously used in the literature to show $S(2)=\pi^2/8,$ which implies Euler's identity $\zeta(2)=\pi^2/6.$ Then, we generalize each integral in order to find the considered sums. The $k$ dimensional analogue of the first integral is the density function of the quotient of $k$ independent, nonnegative Cauchy random variables. In seeking this function, we encounter a special logarithmic integral that we can directly relate to $S(k).$ The $k$ dimensional analogue of the second integral, upon a change of variables, is the volume of a convex polytope, which can be expressed as a probability involving certain pairwise sums of $k$ independent uniform random variables. We use combinatorial arguments to find the volume, which in turn gives new closed formulas for $S(k)$ and $\zeta(2k).$ The $k$ dimensional analogue of the last integral, upon another change of variables, is an integral of the joint density function of $k$ Cauchy random variables over a hyperbolic polytope. This integral can be expressed as a probability involving certain pairwise products of these random variables, and it is equal to the probability from the second generalization. Thus, we specifically highlight the similarities in the combinatorial arguments between the second and third generalizations.
\end{abstract}

\section{Introduction}
One of the most celebrated problems in Classical Analysis is the \textit{Basel Problem}, which is to evaluate the infinite series
\begin{equation}
\zeta(2)=\sum_{n=1}^\infty\frac{1}{n^2}.
\label{euler}
\end{equation}
Pietro Mengoli initially posed this problem in 1644, and Euler \cite{LE} was the first to give the correct answer $\pi^2/6.$ Since Euler's time, however, more solutions to the  problem have appeared in the literature. For example, Kalman \cite{DK} records ones with techniques from different areas such as Fourier Analysis and complex variables. 

In this article, we generalize solutions involving multiple integration to the closely related problem of finding the sums:
\begin{align}\label{S(k)}
S(k) & = \sum_{n=0}^{\infty} \frac{(-1)^{nk}}{(2n+1)^k}, \quad k \in \mathbb{N}, \\  \label{zeta(2k)}
\zeta(2k) & = \sum_{n=1}^{\infty} \frac{1}{n^{2k}}, \quad k \in \mathbb{N},
\end{align}
with the latter sum being the Riemann Zeta Function evaluated at the positive even integers. Apostol \cite{A} pioneered this approach on $\zeta(2).$ He evaluates 
\begin{equation} 
\int_{0}^{1}\int_{0}^{1} \frac{1}{1-xy} \ dx \ dy
\label{Apostol}
\end{equation}
in two ways: first, by converting the integrand into a geometric series and exchanging summation and integration to obtain $\zeta(2),$ and then using the linear change of variables
\[x= \frac{u+v}{\sqrt{2}}, \quad y= \frac{u-v}{\sqrt{2}}\] to obtain $\pi^2/{6}$ upon evaluating arctangent integrals.

We focus on three different double integrals as our starting point:
\begin{equation} \label{Double Integral 1}
D_1=\int_{0}^{\infty} \int_{0}^{\infty} \frac{y}{(x^2y^2+1)(y^2+1)} \ dx \ dy,
\end{equation}
whose variations appear in \cite{BFY,JH,NL,LP,DR},
\begin{equation} \label{Double Integral 2}
D_2=\int_{0}^{1} \int_{0}^{1} \frac{1}{1-x^2y^2} \ dx \ dy,
\end{equation} 
which appears in \cite{BCK,NE}, and 
\begin{equation} \label{Double Integral 3}
D_3=\int_{0}^{1} \int_{0}^{1} \frac{1}{\sqrt{xy}(1-xy)} \ dx \ dy,
\end{equation}
which appears in \cite[p. 9]{ZK}.
Each integral is used to show 
\begin{equation}\label{pi^2/8}
S(2)=\sum_{n=0}^\infty\frac{1}{(2n+1)^2}= \frac{\pi^2}{8}, 
\end{equation} from which we can algebraically obtain $\zeta(2)=\pi^2/6.$

First, $D_1$ evaluates to $\pi^2/4$, but upon reversing the order of integration, it becomes
\begin{equation}\label{leibniz}
\int_0^\infty\frac{\ln(x)}{x^2-1}\ dx,
\end{equation}
which we show is $2S(2).$  We generalize $D_1,$ considering the density function of the quotient of $k$ Cauchy random variables expressed as a multiple integral. Our approach differs from that of Bourgade, Fujita and Yor \cite{BFY}, who also use these same random variables. We evaluate this multiple integral by repeatedly reversing the order of integration and using partial fractions several times. The result is we obtain a logarithmic integral generalizing \eqref{leibniz}, which we can relate directly back to $S(k).$ We then modify our approach to find special closed forms to the bilateral alternating series
\begin{equation} \label{Bilateral Hurwitz Zeta Series}
S(k,a) = \sum_{n=-\infty}^{\infty} \frac{(-1)^{nk}}{(an+1)^k}, \quad a>1, \quad  \frac{1}{a} \notin \mathbb{N}, \quad k \in \mathbb{N},  
\end{equation}
of which the special case $k=2$ is examined in \cite{JC}. 
 
Next, $D_2$, similar to \eqref{Apostol}, can be evaluated in two ways. First, we convert the integrand into a geometric series and exchange summation and integration to obtain $S(2).$ Then, using Calabi's  trigonometric change of variables \cite{BCK} (the hyperbolic version provided in \cite{D'AK,LI,silagadze2010,silagadze2012})
\begin{equation} \label{Calabi COV}
x=\frac{\sin(u)}{\cos(v)},\quad  y=\frac{\sin(v)}{\cos(u)},
\end{equation}
whose Jacobian determinant is $\left|\frac{\partial(x,y)}{\partial(u,v)} \right|=1-x^2y^2,$ 
we see $D_2$ is the area of a right isosceles triangle with the shorter legs each of length $\pi/2.$ Hence,  $D_2=\pi^2/8.$ The generalized version of this approach leads us to a convex polytope 
\begin{equation}\label{Polytope}
\Delta^k= \lbrace (u_1, \dots, u_k) \in \mathbb{R}^k : u_{i}+u_{i+1}<1, \ u_i>0, 1 \leq i \leq k \rbrace,
\end{equation} 
in which we use cyclical indexing mod $k:$ $u_{k+1}:=u_{1}.$  The volume of $\Delta^k,$ which is equal to $(2/\pi)^k S(k),$ has already been computed in three different ways: Beukers, Calabi, and Kolk \cite{BCK} dissect $\Delta^k$ into pairwise disjoint congruent pyramids, Elkies \cite{NE} and Silagadze \cite{silagadze2012} both perform spectral analysis on its characteristic function, and Stanley \cite{RS} uses properties of alternating permutations. We give another approach to the volume computation, viewing it as the probability that $k$ independent uniform random variables on $(0,1)$ have cyclically pairwise consecutive sums less than $1.$ We use combinatorial arguments to evaluate the probability. Our approach leads to new and interesting closed formulas of $S(k)$ and $\zeta(2k)$ that do not invoke the traditional Bernoulli and Eulerian numbers.

Finally, upon a substitution $x=\sqrt{u}, y=\sqrt{v}$ to $D_2$, we see that $D_3=4S(2).$ On the other hand, Zagier's and Kontsevich's change of variables \cite[p. 9]{ZK}
\begin{equation} \label{Zagier COV}
x= \frac{\xi^2(\eta^2+1)}{\xi^2+1}, \quad y= \frac{\eta^2(\xi^2+1)}{\eta^2+1},
\end{equation}
which has Jacobian Determinant $\left| \frac{\partial(x,y)}{\partial(\xi,\eta)} \right| = \frac{4\sqrt{xy}(1-xy)}{(\xi^2+1)(\eta^2+1)},$ transforms $D_3$ into 
\[\iint_{\substack{\xi \eta<1, \\ \xi, \eta>0}}\frac{4}{(\xi^2+1)(\eta^2+1)}\ d \eta \ d \xi,\] whose value is $\pi^2/2.$ We generalize this approach, which leads us to integrating the joint density function of $k$ independent, nonnegative Cauchy random variables over a hyperbolic polytope:
\begin{equation}\label{Hypertope}
\mathbb{H}^k= \lbrace (\xi_1, \dots, \xi_k) \in \mathbb{R}^k : \xi_i\xi_{i+1}<1, \ \xi_i>0, 1 \leq i \leq k \rbrace,
\end{equation} 
in which we use cyclical indexing mod $k:$ $\xi_{k+1}:=\xi_1.$ This is the same as the probability that $k$ independent, nonnegative Cauchy random variables have cyclically pairwise consecutive products less than $1.$ Combinatorial analysis of this probability leads to the exact same closed formulas from the second approach, but we highlight the similarities between this analysis and that of our second approach. Hence, this approach, along with the second, induces two equivalent probabilistic viewpoints of $S(k)$ and $\zeta(2k).$  
 
Finally, it is worth noting that $\zeta(2k)$ can be obtained from $S(2k)$ by observing 
\[\zeta(2k)=\frac{1}{2^{2k}} \zeta(2k)+ S(2k),\]
which implies
\begin{equation} \label{Zeta(2k) in terms of S(2k)}
\zeta(2k)=\frac{2^{2k}}{2^{2k}-1} S(2k).
\end{equation}

\subsection{Some Concepts From Probability Theory}

We briefly recall some notions from probability theory that we use throughout this article.

Let $A_1, \dots , A_k$ be $k$ events. We denote $\text{Pr}(A_1,\dots, A_k)$ to be the probability that $A_1,\dots , A_k$ occur at the same time.   

Let $X_1,\dots ,X_k$ be $k$ continuous, nonnegative random variables. Their \textit{joint density function} $f_{X_1, \dots ,X_k}(x_1, \dots ,x_k)$ is a nonnegative function such that
\begin{equation}\label{Probability Axiom 1}
\int_{0}^{\infty} \dots\int_{0}^{\infty} f_{X_1, \dots ,X_k}(x_1, \dots ,x_k) \ dx_1 \dots dx_k = 1,
\end{equation}
and for all $a_1,\, b_1,\, \dots ,\,a_k,\, b_k \geq 0,$ we have 
\begin{equation} \label{Probability Axiom 2} 
\text{Pr}\left(a_1 \leq X_1 \leq b_1, \dots , a_k \leq X_k \leq b_k \right)  = \\  \int_{a_k}^{b_k} \dots\int_{a_1}^{b_1} f_{X_1, \dots ,X_k}(x_1, \dots ,x_k) \ dx_1 \dots dx_k.
\end{equation}
The \textit{joint cumulative distribution function} of $X_1, \dots ,X_k$ is the function 
\begin{equation} \label{Joint CDF}
F_{X_1, \dots, X_k}(x_1, \dots, x_k)=\text{Pr}(X_1 \leq x_1, \dots , X_k \leq x_k).
\end{equation} 
In the case there is one random variable $X_1$, we simply refer to $f_{X_1}(x_1)$ as its \textit{density function} and $F_{X_1}(x_1)$ as its \textit{cumulative distribution function}. It follows from the Fundamental Theorem of Calculus that 
\begin{equation} \label{Differentiate CDF}
\frac{\partial^{k}}{\partial x_1 \dots \partial x_k} F_{X_1, \dots ,X_k} = f_{X_1, \dots, X_k},
\end{equation}
and in the case of only one variable $X_1,$ we have
$F'_{X_1}  = f_{X_1}.$

If $X_1, \dots, X_k$ are \textit{independent}, then we have 
\begin{equation}\label{Joint CDF Independence}
F_{X_1, \dots, X_k}(x_1, \dots x_k) = F_{X_1}(x_1) \dots F_{X_k}(x_k), 
\end{equation}
and from \eqref{Differentiate CDF}, we have
\begin{equation}\label{Joint PDF Independence}
f_{X_1, \dots, X_k}(x_1, \dots x_k) = f_{X_1}(x_1) \dots f_{X_k}(x_k).
\end{equation}

Let $X$ and $Y$ be independent, nonnegative random variables with density functions $f_X(x)$ and $f_Y(y),$ respectively. Their \textit{product} $Z=XY$ is a new random variable with density function
\begin{equation} \label{Product}
f_Z(z)= \int_{0}^{\infty}\frac{1}{x}f_Y\left(\frac{z}{x}\right)f_X(x) \ dx.
\end{equation}
If $X \neq 0,$ their \textit{quotient} $T=Y/X$ has density function
\begin{equation}\label{Quotient}
f_T(t)=\int_{0}^\infty x f_Y(tx)f_X(x) \ dx.
\end{equation} 
Random variable operations are studied in Springer's book \cite{springer1979}. 

Finally, the nonnegative random variable $X$ is said to be  \textit{Cauchy} if
\begin{equation}\label{Cauchy RV}
f_{X}(x) = \frac{2}{\pi} \frac{1}{x^2+1},
\end{equation}
and $X$ is said to be \textit{uniform} on $(a,b)$ with $a<b$ if
\begin{equation}\label{Uniform RV}
f_{X}(x) = \frac{1}{b-a}.
\end{equation} 
\section{Cauchy Random Variables} 
We first give a slightly modified version of the solution to the Basel Problem given by Pace \cite{LP}  (see also \cite{BFY}), which we will then generalize. Using the results of our generalization, we highlight two specific applications, in which we show $S(3)=\pi^3/32$ and $S(4)=\pi^4/96.$ The latter result gives a new proof of $\zeta(4)=\pi^4/90$ with multiple integration. Finally, we extend our generalization to find $S(k,a)$ and provide closed formulas for the cases $k=2$ and $3.$   

\subsection{Luigi Pace's Solution to the Basel Problem }

Let $X_1$ and $X_2$ be independent, nonnegative Cauchy random variables. Define
\[Z=X_1/X_2.\]
We compute the density function $f_{Z}(z).$  By \eqref{Quotient} and \eqref{Cauchy RV}, we have 
\begin{equation}\label{Joint Density 2 Cauchy RVS}
f_{Z}(z)  =\frac{4}{\pi^2} \int_{0}^{\infty} \frac{x_2}{(z^2x^2_2+1)(x^2_2+1)} \ dx_2. 
\end{equation} 

To evaluate \eqref{Joint Density 2 Cauchy RVS}, we use the partial fractions identity
\begin{equation}\label{Partial Fractions Identity}
\frac{x}{(x^2y^2+1)(x^2+1)} =\frac{xy^2}{(x^2y^2+1)(y^2-1)} -\frac{x}{(x^2+1)(y^2-1)}.
\end{equation}
Hence, 
\[f_{Z}(z) = \frac{4}{\pi^2} \frac{\ln(z)}{z^2-1}.\]
Using \eqref{Probability Axiom 1} and then rearranging terms, we have 
\begin{equation}\label{pi^2/4}
\int_{0}^{\infty} \frac{\ln(z)}{z^2-1} \ dz  = \frac{\pi^2}{4}.
\end{equation}

Following the argument of \cite{DR}, we write
\begin{align}
\int_{0}^{\infty} \frac{\ln(z)}{z^2-1} & =  \int_{0}^{1} \frac{\ln(z)}{z^2-1} \ dz +\int_{1}^{\infty} \frac{\ln(z)}{z^2-1} \ dz \label{split}  \\ 
& = 2 \int_{0}^{1} \frac{\ln(z)}{z^2-1} \ dz \label{combine},
\end{align}
in which \eqref{combine} follows from a substitution $z=1/u$ to the second term on the right hand side of \eqref{split}. We expand the integrand of \eqref{combine} into a geometric series 
\begin{equation}\label{Geometric Series}
\frac{\ln(z)}{z^2-1} =-\sum_{n=0}^{\infty}z^{2n} \ln(z).
\end{equation}
Putting the right hand side of \eqref{Geometric Series} in place of the integrand in \eqref{combine}, we have \eqref{combine} is equal to 
\begin{align}
2\int_{0}^{1} -\sum_{n=0}^{\infty}z^{2n} \ln(z) \ dz & = -2\sum_{n=0}^{\infty} \int_{0}^{1} z^{2n} \ln(z) \ dz \label{Monotone Convergence Thm} \\
& = 2 \sum_{n=0}^{\infty} \frac{1}{(2n+1)^2} \label{IBP},
\end{align}
by which \eqref{Monotone Convergence Thm} follows by the Monotone Convergence Theorem  (see \cite[p. 95-96]{CK}), and \eqref{IBP} follows upon integration by parts. Hence, we have  
\[S(2) = \frac{\pi^2}{8}.\]
Finally, using \eqref{Zeta(2k) in terms of S(2k)}, we obtain
\[\zeta(2) = \frac{4}{3} S(2) = \frac{\pi^2}{6}.\]
Thus, the Basel Problem is solved.
\subsection{Generalization for $S(k)$}
We now give the generalization of Pace's solution, as well as all the solutions presented in \cite{JH,NL,LP,DR}.

Let $X_1, \dots , X_k$ be $k$ independent, nonnegative Cauchy random variables with $k \geq 2.$ Define the quotient 
\[Z_k= X_1/\dots/ X_k.\]
We seek the density function $f_{Z_k}(z).$ 
\begin{theorem} \label{Joint Density Quotient k Cauchy RVs Thm}
The density function $f_{Z_k}(z)$ is
\begin{equation}\label{Joint Density Quotient k Cauchy RVs}
f_{Z_k}(z) = \frac{2^k}{\pi^k}\int_{0}^{\infty} \dots \int_{0}^{\infty} \frac{x_2 \dots x_k}{(z^2 x^2_2 \dots x^2_k + 1)(x^2_2+1) \dots (x^2_k+1)} \   dx_2 \dots dx_k.  
\end{equation}
\end{theorem}
\begin{proof}
The case $k=2$ was already computed in \eqref{Joint Density 2 Cauchy RVS}. 

Let the statement hold for $k=m$ with $m>2.$ We show it must also hold for $k=m+1.$ Applying \eqref{Quotient} to $Z_{m+1}=Z_{m}/X_{m+1},$ we have 
\begin{align}
f_{Z_{m+1}}(z)  & = \frac{2}{\pi} \int_{0}^{\infty} x_{m+1}f_{Z_m}(zx_{m+1})f_{X_{m+1}}(x_{m+1}) \ dx_{m+1}\nonumber \\ 
& = \frac{2^{m+1}}{\pi^{m+1}} \int_{0}^{\infty} \dots \int_{0}^{\infty} \frac{x_2 \dots x_m x_{m+1}}{(z^2 x^2_2 \dots x^2_{m+1} +1)(x^2_2+1) \dots (x^2_{m+1}+1)} \   dx_2 \dots dx_{m+1} \label{(m+1)th Integral},
\end{align}
in which \eqref{(m+1)th Integral} is the result of the inductive hypothesis on $f_{Z_m}(z).$ 

Note that the integrand of \eqref{(m+1)th Integral} is nonnegative. Thus, Tonelli's Theorem allows us to reverse the order of integration on the integral  
\begin{equation} \label{Induction Integral}
\int_{0}^{\infty} f_{Z_{m+1}}(z) \ dz.
\end{equation}
Integrating \eqref{Induction Integral} with respect to $z$ first, and then with respect to each of the other $m$ variables, we see \eqref{Induction Integral} is equal to $1,$ satisfying \eqref{Probability Axiom 1}.   
\end{proof}

\begin{remark*} 
Theorem \ref{Joint Density Quotient k Cauchy RVs Thm} still holds if $Z_k$ is alternatively formulated by the quotient of $X_1$ and the product $X_2 \dots X_k.$ In this case, we would need to use \eqref{Product} and \eqref{Quotient}.
\end{remark*}

The crux of our generalization is evaluating \eqref{Joint Density Quotient k Cauchy RVs}, which involves reversing the order of integration several times and using a generalized partial fractions identity
\begin{equation} \label{Generalized PF}
\frac{x}{(x^2y^2-(-1)^s)(x^2+1)} = \frac{xy^2}{(x^2y^2-(-1)^s)(y^2-(-1)^s)}-\frac{x}{(x^2+1)(y^2-(-1)^s)}, \quad s \in \mathbb{N}.
\end{equation}
Along the way, we encounter integrals of the form
\begin{equation}\label{Cauchy PV 0}
\int_{0}^{\infty} \frac{\ln^{k}(z)}{z^2-(-1)^k} \ dz,
\end{equation}
which vanish, according to Gradshteyn and Rhyzik table entries \cite[Section 4.271: 7 and 9]{GH}. The end result is that we arrive at a logarithmic integral of the form
\begin{equation}\label{Logarithmic k Integral}
J_k = \int_{0}^{\infty} \frac{\ln^{k-1}(z)}{z^2-(-1)^k} \ dz.
\end{equation}
By splitting the region of integration in the same way as in \eqref{split} and making a substitution $z=1/u$ to the resulting second term, we find that
\begin{equation} \label{J_k}
J_k = 2 \int_{0}^{1} \frac{\ln^{k-1}(z)}{z^2-(-1)^k} \ dz.
\end{equation}
From entries \cite[Section 4.271: 6 and 10]{GH}, we have
\begin{equation} \label{Euler Bernoulli Numbers}
J_k = \begin{cases}
\dfrac{2^{2k}-2^k}{k}  \left( \dfrac{\pi}{2} \right)^{k} |B_{k}| & k \text{ even}\\[2mm]

\left( \dfrac{\pi}{2} \right)^{k} |E_{k-1}| & k \text{ odd} 
\end{cases}.
\end{equation}
Here, $B_m$ and $E_m$ denote the Bernoulli number and Eulerian number of order $m,$ respectively. These numbers satisfy  
\[\frac{x}{e^x-1} =\sum_{m=0}^{\infty}\frac{B_m}{m!}x^m, \quad  \frac{1}{\cosh(x)}=\sum_{m=0}^{\infty}\frac{E_m}{m!}x^m.\]
It is worth noting $B_m$ and $E_m$ can be alternatively defined through the generating functions of $\tan(x)$ and $\sec(x),$ respectively (see \cite{RS}).
Finally, we convert the integrand in \eqref{J_k} into a geometric series
\begin{equation} \label{S(k) General Geometric Series}
\frac{\ln^{k-1}(z)}{z^2-(-1)^k}= -  \sum_{n=0}^{\infty} \ln^{k-1}(z)  (-1)^{nk} z^{2n}.
\end{equation}
We put the series in \eqref{S(k) General Geometric Series} in place of the original integrand in \eqref{J_k}, then exchange summation and integration as permitted by the Monotone Convergence Theorem, and finally use repeated integration by parts to obtain the result
\begin{equation} \label{S(k) J_k relation}
S(k)=\frac{J_k}{2(k-1)!}. 
\end{equation}

\subsection{Evaluating $S(3)$ and $S(4)$}
We use our density function results to evaluate $S(3)$ and $S(4),$ with the latter giving $\zeta(4)=\pi^4/{90}.$ 
\begin{theorem} \label{S(3)}
The value of $S(3)$ is
\[S(3)=\sum_{n=0}^{\infty} \frac{(-1)^n}{(2n+1)^3}=\frac{\pi^3}{32}.\]
\end{theorem}
\begin{proof}
Let $X_1, X_2,$ and $X_3$ be independent, nonnegative Cauchy random variables, and let $Z_3$ be their quotient. Then, by Theorem \ref{Joint Density Quotient k Cauchy RVs Thm}, we have the density function 
\[f_{Z_3}(z)= \frac{8}{\pi^3} \int_{0}^{\infty} \int_{0}^{\infty}  \frac{x_2x_3}{(z^2x_2^2x_3^2+1)(x_2^2+1)(x_3^2+1)}  \ dx_2 \ dx_3. \]

Using the partial fractions identity in \eqref{Generalized PF}, we have 
\begin{align}
f_{Z_3}(z) & =\frac{8}{\pi^3} \int_{0}^{\infty}  \frac{\ln(zx_3)}{(z^2 x^2_3-1)(x^2_3+1)} \ dx_3 \nonumber\\
& =\frac{8}{\pi^3} \int_{0}^{\infty} \frac{x_3\ln(z)}{(z^2x^2_3-1)(x^2_3+1)} + \frac{8}{\pi^3} \int_{0}^{\infty} \frac{x_3 \ln(x_3)}{(z^2x^2_3-1)(x^2_3+1)}  dx_3 \label{S(3) split}.
\end{align}
Using \eqref{Generalized PF} again on the first term of \eqref{S(3) split}, we see 
\begin{equation}\label{S(3) combine}
f_{Z_3}(z)= \frac{8}{\pi^3} \frac{\ln^2(z)}{z^2+1} +\frac{8}{\pi^3} \int_{0}^{\infty} \frac{x_3\ln(x_3)}{(z^2x^2_3-1)(x^2_3+1)} \ dx_3.
\end{equation}

Now, by \eqref{Probability Axiom 1}, we have
\begin{equation} \label{Z_3}
1= \frac{8}{\pi^3} J_3 + \frac{8}{\pi^3} \int_{0}^{\infty} \int_{0}^{\infty} \frac{x_3\ln(x_3)}{(z^2x^2_3-1)(x^2_3+1)} \ dx_3 \ dz.
\end{equation}
Reversing the order of integration on the second term in \eqref{Z_3}, we encounter an integral of the form examined in \eqref{Cauchy PV 0}, so this term vanishes. Thus, rearranging the terms in \eqref{Z_3} gives
\[J_3 = \frac{\pi^3}{8}.\]
Applying \eqref{S(k) J_k relation}, we find
\[
S(3) = \sum_{n=0}^{\infty}\frac{(-1)^n}{(2n+1)^3}=\frac{1}{2(2!)} \frac{\pi^3}{8} = \frac{\pi^3}{32}.
\]
\end{proof}

\begin{remark*}
Our result for $J_3$ agrees with \eqref{Euler Bernoulli Numbers}, since $|E_2|=1.$
\end{remark*}

We now find $S(4)$ and $\zeta(4).$

\begin{theorem} \label{S(4) Zeta(4)}
The values of $S(4)$ and $\zeta(4)$  are
\begin{align*}
S(4) & =\sum_{n=0}^{\infty} \frac{1}{(2n+1)^4} = \frac{\pi^4}{96} \\
\zeta(4) & = \sum_{n=1}^{\infty} \frac{1}{n^4} = \frac{\pi^4}{90}.
\end{align*}

\end{theorem}

\begin{proof}
Let $X_1,\dots,X_4$ be independent, nonnegative Cauchy random variables.  Defining $Z_4$ as their quotient, Theorem \ref{Joint Density Quotient k Cauchy RVs Thm} implies  
\[f_{Z_4}(z) =\frac{16}{\pi^4} \int_{0}^{\infty}\int_{0}^{\infty}\int_{0}^{\infty}  \frac{1}{z^2x^2_2x^2_3x^2_4+1} \prod_{i=2}^{4}\frac{x_i}{x^2_i+1} \ dx_2 \ dx_3 \ dx_4 .\]

Applying the partial fractions identity from \eqref{Generalized PF}, we have 
\[f_{Z_4}(z) = \frac{16}{\pi^4} \int_{0}^{\infty} \int_{0}^{\infty}  \frac{x_3x_4\ln(zx_3x_4)}{(z^2x^2_3x^2_4-1)(x^2_3+1)(x^2_4+1)} \ dx_3 \ dx_4,\]
which splits into three terms: 
\begin{align}
\label{42} H_1=\frac{16}{\pi^4} \int_{0}^{\infty} \int_{0}^{\infty}  \frac{x_3x_4\ln(z)}{(z^2x^2_3x^2_4-1)(x^2_3+1)(x^2_4+1)} \ dx_3 \ dx_4 , \\ 
\label{43} H_2=\frac{16}{\pi^4} \int_{0}^{\infty} \int_{0}^{\infty}  \frac{x_3x_4\ln(x_3)}{(z^2x^2_3x^2_4-1)(x^2_3+1)(x^2_4+1)} \ dx_3 \ dx_4 , \\
\label{44} H_3=\frac{16}{\pi^4} \int_{0}^{\infty} \int_{0}^{\infty}  \frac{x_3x_4\ln(x_4)}{(z^2x^2_3x^2_4-1)(x^2_3+1)(x^2_4+1)} \ dx_3 \ dx_4.
\end{align}

We first evaluate $H_1.$  Using \eqref{Generalized PF}, we see
\[H_1 =\frac{16}{\pi^4} \int_{0}^{\infty}   \frac{x_4\ln^2(z)}{(z^2x^2_4+1)(x^2_4+1)}  +   \frac{16}{\pi^4} \int_{0}^{\infty} \frac{x_4\ln(z)\ln(x_4)}{(z^2x^2_4+1)(x^2_4+1)} \ dx_4. \]
Using \eqref{Generalized PF} on the first term on the right hand side, we have 
\[H_1 = \frac{16}{\pi^4} \frac{\ln^3(z)}{z^2-1}  + \frac{16}{\pi^4} \int_{0}^{\infty}  \frac{x_4\ln(z)\ln(x_4)}{(z^2x^2_4+1)(x^2_4+1)} \ dx_4.\]
Next, for $H_2$, we reverse the order of integration and use \eqref{Generalized PF} to get
\[H_2=\frac{16}{\pi^4} \int_{0}^{\infty}  \frac{x_3\ln(x_3)}{(z^2x^2_3+1)(x^2_3+1)} \ dx_3. \]
Similarly, for $H_3$, we get
\[H_3=\frac{16}{\pi^4} \int_{0}^{\infty}  \frac{x_4\ln(x_4)}{(z^2x^2_4+1)(x^2_4+1)} \ dx_4. \]

Using \eqref{Probability Axiom 1}, we have 
 
\begin{multline}\label{Z_4}
1 =  \frac{16}{\pi^4} J_4  + \frac{16}{\pi^4} \int_{0}^{\infty} \int_{0}^{\infty}  \frac{x_4\ln(z)\ln(x_4)}{(z^2x^2_4+1)(x^2_4+1)} \ dx_4 \ dz \\ + \frac{16}{\pi^4} 
 \int_{0}^{\infty} \int_{0}^{\infty} \frac{x_3\ln(x_3)}{(z^2x^2_3+1)(x^2_3+1)} \ dx_3 \ dz \\+ \frac{16}{\pi^4} \int_{0}^{\infty} \int_{0}^{\infty}  \frac{x_4\ln(x_4)}{(z^2x^2_4+1)(x^2_4+1)} \ dx_4 \ dz.
\end{multline}
The third term and fourth terms on the right hand side of \eqref{Z_4} vanish, which is seen upon reversing the order of integration in each integral and observing each inner integral is of the form in \eqref{Cauchy PV 0}. To evaluate the second term, we make the substitution $z=t/{x_4},$ which transforms it into
\begin{equation}\label{Second Term}
\frac{16}{\pi^4} \int_{0}^{\infty} \int_{0}^{\infty}  \frac{\ln(t)\ln(x_4)}{(t^2+1)(x^2_4+1)} \ dt \ dx_4 -\frac{16}{\pi^4} \int_{0}^{\infty} \int_{0}^{\infty} \frac{\ln^2(x_4)}{(t^2+1)(x^2_4+1)} \ dt \ dx_4.
\end{equation}
The first term of \eqref{Second Term} is of the same form seen in \eqref{Cauchy PV 0}, but we recognize the second term is $-8J_3/{\pi^3}=-1.$ Thus, we see 
\[1=\frac{16}{\pi^4} J_4 -1,\]
so rearranging terms gives  
\[J_4= \frac{\pi^4}{8}.\]
By \eqref{S(k) J_k relation}, we have 
\[S(4) = \sum_{n=0}^{\infty} \frac{1}{(2n+1)^4} = \frac{1}{2(3!)}\frac{\pi^4}{8} = \frac{\pi^4}{96}.\]
Finally, from \eqref{Zeta(2k) in terms of S(2k)}, we have
\[\zeta(4) = \sum_{n=1}^{\infty} \frac{1}{n^4} = \frac{16}{15}\left(\frac{\pi^4}{96}\right) = \frac{\pi^4}{90}.\]
\end{proof}
\subsection{Further Generalization to $S(k,a)$}
We recall from the alternating series from \eqref{Bilateral Hurwitz Zeta Series}, which is 
\[S(k,a) =\sum_{n=-\infty}^{\infty} \frac{(-1)^{nk}}{(an+1)^k}, \quad a>1, \quad  \frac{1}{a} \notin \mathbb{N}, \quad k \in \mathbb{N}.\]
We consider $k$ independent, nonnegative random variables $X_1,\dots ,X_k,$ each with density function
\begin{equation} \label{Generalized Cauchy PDF}
f_{X_i}(x_i) = 
\begin{cases} \dfrac{a}{\pi}\,\sin \left( \dfrac{\pi}{a}\right)\, \dfrac{1}{x^a_i+1} & i=1 \\
\\
\dfrac{2}{\pi}\,\sin \left( \dfrac{\pi}{a}\right)\, \dfrac{x_i^{1-\frac{2}{a}}}{x^2_i+1} & 2 \leq i \leq k 
\end{cases},
\end{equation}
where $a$ has the same conditions that we imposed on $S(k,a).$ 

\begin{theorem} \label{Generalized PDF Theorem}
All of $f_{X_1}(x_1), \dots, f_{X_k}(x_k)$ are valid density functions.   
\end{theorem}
\begin{proof}
The crux of the proof is the application of Euler's Reflection Formula:
\begin{equation} \label{Euler Reflection Formula}
\Gamma(t)\Gamma(1-t)=\pi \csc(\pi t), \quad t>0, \quad t \notin \mathbb{N},
\end{equation}
where  
\[\Gamma(t)=\int_{0}^{\infty} x^{t-1} e^{-x} \ dx.\]

For $i=1,$ we see
\begin{align}
\int_{0}^{\infty} f_{X_1}(x_1) \ dx_1  & =  \frac{a}{\pi} \sin \left( \frac{\pi}{a}\right) \int_{0}^{\infty} \frac{1}{x^a_1+1} \ dx_1  \nonumber \\
& = \frac{a}{\pi} \sin \left( \frac{\pi}{a}\right) \int_{0}^{\infty} \int_{0}^{\infty} e^{-y(1+x^a_1)} \ dy \   dx_1 \label{Substitution into Gamma Product}  \\ 
& = \frac{a}{\pi} \sin \left( \frac{\pi}{a}\right)  \int_{0}^{\infty} t^{\frac{1}{a}-1} e^{-z} \ dt \ \int_{0}^{\infty} \frac{ y^{\frac{1}{a}} e^{-y}}{a} \ dy \label{Gamma Product Integrals} \\
& = \frac{a}{\pi} \sin \left( \frac{\pi}{a}\right) \frac{\Gamma \left(\frac{1}{a} \right) \Gamma \left(1-\frac{1}{a} \right)}{a} \nonumber \\ 
& = 1 \nonumber,
\end{align}
by which we made the substitution $x_1=\left(t/y \right)^{1/a}$ on \eqref{Substitution into Gamma Product} to obtain \eqref{Gamma Product Integrals}.

For $i>1,$ we have  
\begin{align}
\int_{0}^{\infty} f_{X_i}(x_i) \ dx_i 
& = \frac{2}{\pi} \sin \left( \frac{\pi}{a}\right) \int_{0}^{\infty} \frac{x_i^{1-\frac{2}{a}}}{x^2_i+1} \ dx_i \nonumber \\
& = \frac{2}{\pi} \sin \left( \frac{\pi}{a}\right) \int_{0}^{\infty} \int_{0}^{\infty} x_i^{1-\frac{2}{a}} e^{-y(1+x^2_i)} \ dy \   dx_i \label{40} \\
& = \frac{2}{\pi} \sin \left(\frac{\pi}{a} \right) \frac{\Gamma(\frac{1}{a})\Gamma(1-\frac{1}{a})}{2} \label{41}\\
& = 1 \nonumber , 
\end{align}
in which \eqref{41} follows from the substitution $x_i=\sqrt{t/y}$ on \eqref{40}.   
\end{proof}

\begin{remark*}
When $a=2,$ we see $X_1, \dots, X_k$ are Cauchy.
\end{remark*}

Define the random variable for $k \geq 2,$ 
\[Z_{k,a}=X_1 / (X_2 \dots X_k)^{\frac{2}{a}}.\] We seek the density function $f_{Z_{k,a}}(z).$ 

\begin{theorem} \label{Joint Density Generalized k Cauchy RVs Thm}
The density function $f_{Z_{k,a}}(z)$ is 
\begin{equation} \label{Joint Density Generalized k Cauchy RVs}
f_{Z_{k,a}}(z) =  \frac{a  \left(2\sin \left( \frac{\pi}{a} \right) \right)^k}{2 \pi^k} \int_{0}^{\infty}...\int_{0}^{\infty} \frac{ \ x_2 \dots x_k}{(z^a x^2_2 \dots x_k^2 + 1)(x^2_2+1) \dots (x^2_k+1)} \ dx_2 \dots dx_k.
\end{equation}
\end{theorem}

\begin{proof}
We start with the cumulative distribution function
\begin{equation} \label{CDF Z_{k,a}}
F_{Z_{k,a}}(z) = \text{Pr}(Z_{k,a} \leq z).
\end{equation}
Expanding the right hand side of \eqref{CDF Z_{k,a}}, we find
\begin{align*}
\text{Pr}(Z_{k,a} \leq z) & = \text{Pr}\left(X_1\leq z (X_2 \dots X_k)^{\frac{2}{a}} \right) \\
& = F_{X_1}\left(z (X_2 \dots X_k)^{\frac{2}{a}}\right) \\
& = \frac{a  \left(2\sin \left( \frac{\pi}{a} \right) \right)^k}{2 \pi^k} \int_{0}^{\infty}  \dots \int_{0}^{\infty}  \int_{0}^{z(x_2 \dots x_k)^\frac{2}{a}} \frac{(x_2 \dots x_k)^{1-\frac{2}{a}}}{(x_1^a + 1)(x^2_2+1) \dots (x^2_k+1)}  \ dx_1 \ dx_2 \dots \ dx_k. 
\end{align*}
Thus, by \eqref{Differentiate CDF}, we have
\[f_{Z_{k,a}}(z) = \frac{a  \left(2\sin \left( \frac{\pi}{a} \right) \right)^k}{2 \pi^k} \int_{0}^{\infty}...\int_{0}^{\infty} \frac{ \ x_2 \dots x_k}{(z^a x^2_2 \dots x_k^2 + 1)(x^2_2+1) \dots (x^2_k+1)} \ dx_2 \dots dx_k.\]

Now, we consider 
\begin{equation} \label{Integral PDF}
\int_{0}^{\infty} f_{Z_{k,a}}(z) \ dz,
\end{equation}
whose integrand is nonnegative. Hence, Tonelli's Theorem allows us to reverse the order of integration in \eqref{Integral PDF}. We integrate with respect to $z$ first and then with respect to the other $k-1$ variables. Making the substitution $z= t(x_2 \dots x_k)^{-2/a},$ and then applying Theorem \ref{Generalized PDF Theorem} several times, we see \eqref{Integral PDF} is equal to $1,$ satisfying \eqref{Probability Axiom 1}. 
\end{proof}

\begin{remark*}
When $a=2$, we recover $f_{Z_k}(z)$ from \eqref{Joint Density Quotient k Cauchy RVs}. 
\end{remark*}

In simplifying the density function when $k>2,$ we encounter integrals of the form 
\begin{equation} \label{Generalized Cauchy PV}
\int_{0}^{\infty} \frac{t^{m-1}}{t^n-1} \ dt = -\frac{\pi}{n} \cot\left(\frac{m\pi}{n}\right), \quad m<n,
\end{equation}
as listed in Gradshteyn and Rhyzik entry \cite[Section 3.241: 4]{GH}. At the end, we arrive at the logarithmic integral
\begin{equation} \label{J_(k,a)}
J_{k,a}= \int_{0}^{\infty} \frac{\ln^{k-1}(z)}{z^a-(-1)^k} \ dz, 
\end{equation}
which generalizes $J_k$ from \eqref{J_k}. We now highlight the relation between $J_{k,a}$ and $S(k,a).$  

\begin{theorem} \label{S(k,a) J(k,a) Relation Theorem}
The following relation holds.
\begin{equation} \label{S(k,a) J(k,a) Relation}
S(k,a)=\frac{J_{k,a}}{(k-1)!}.
\end{equation}

\end{theorem}
\begin{proof}
We write
\begin{align}
J_{k,a} & =\int_{0}^{1} \frac{\ln^{k-1}(z)}{z^{a}-(-1)^k} \ dz + \int_{1}^{\infty} \frac{\ln^{k-1}(z)}{z^{a}-(-1)^k} \ dz  \label{Split integrals} \\
& = \int_{0}^{1} \frac{\ln^{k-1}(z)}{z^{a}-(-1)^k} \ dz + \int_{0}^{1} \frac{\ln^{k-1}(u)}{u^{-a}-(-1)^k} \frac{1}{u^2} \ du \label{z=1/u}  \\
& = (k-1)! \sum_{n=0}^{\infty} \frac{(-1)^{nk}}{(an+1)^k} + (k-1)! \sum_{n=1}^{\infty} \frac{(-1)^{nk}}{(1-an)^k} \label{summation parts} \\
& =  (k-1)! \sum_{n=0}^{\infty} \frac{(-1)^{nk}}{(an+1)^k}+  (k-1)! \sum_{n=-\infty}^{-1} \frac{(-1)^{nk}}{(an+1)^k} \nonumber \\
& =  (k-1)! S(k,a) \nonumber,
\end{align}
in which \eqref{z=1/u} is the result of a substitution $z=1/u$ performed on the second term in \eqref{Split integrals}, and \eqref{summation parts} is the result of converting each integrand in \eqref{z=1/u} into a geometric series, exchanging summation and integration, and finally integrating by parts. 
\end{proof}

\subsection{Evaluating $S(2,a)$ and $S(3,a)$}

With this machinery, we seek the closed forms of $S(2,a)$ and $S(3,a),$ which have manageable calculations. The former sum is a standard complex variables result. Specific examples of $S(2,a)$ for different $a$ are provided in \cite{JC}. 

\begin{theorem} 
$S(2,a)$ is given by the formula
\begin{equation}\label{S(2,a) Result}
S(2,a)= \sum_{n=-\infty}^{\infty} \frac{1}{(an+1)^2}=\frac{\pi^2}{a^2} \csc^{2} \left(\frac{\pi}{a}\right).
\end{equation}

\begin{proof}
Let $X_1$ and $X_2$ be independent, nonnegative random variables defined in \eqref{Generalized Cauchy PDF}, and let $Z_{2,a}=X_1/(X_2)^{2/a}.$ Then,
\[f_{Z_{2,a}}(z)= \frac{2a\sin^{2}(\frac{\pi}{a})}{\pi^2}\int_{0}^{\infty} \frac{x_2}{(z^ax^2_2+1)(x^2_2+1)} \ dx_2,\] from Theorem \ref{Joint Density Generalized k Cauchy RVs Thm}.

Observing $z^a=\left(z^{a/2}\right)^{2},$ we can use \eqref{Generalized PF} to see that
\[f_{Z_{2,a}}(z) =\frac{2a  \sin^{2} \left( \frac{\pi}{a} \right)}{\pi^2}  \frac{\ln(z^{\frac{a}{2}})}{z^a-1} = \frac{a^2  \sin^{2} \left( \frac{\pi}{a} \right)}{\pi^2}  \frac{\ln(z)}{z^a-1}. \] 
Thus, using \eqref{Probability Axiom 1} and then rearranging terms, we see
\[J_{2,a}=\frac{\pi^2}{a^2} \csc^{2} \left(\frac{\pi}{a}\right).\]
Finally, we use Theorem \ref{S(k,a) J(k,a) Relation Theorem} to see 
\[S(2,a) = \sum_{n=-\infty}^{\infty} \frac{1}{(an+1)^2}
= \frac{1}{(2-1)!} J_{2,a} = \frac{\pi^2}{a^2} \csc^{2} \left(\frac{\pi}{a}\right).\]
\end{proof}
\end{theorem}

\begin{remark*}
The case $a=2$ recovers Pace's Basel Problem solution.   
\end{remark*}

Now we find $S(3,a).$
\begin{theorem} 
$S(3,a)$ is given by the formula
\begin{equation}\label{S(3,a) Result}
S(3,a)= \sum_{n=-\infty}^{\infty} \frac{(-1)^n}{(an+1)^3}=\frac{\pi^3}{2a^3} \csc^{3} \left(\frac{\pi}{a}\right)+ \frac{\pi^3}{2a^3} \cot^{2} \left(\frac{\pi}{a}\right)\csc \left(\frac{\pi}{a}\right).
\end{equation}
\end{theorem}
\begin{proof}
Let $X_1, X_2,$ and $X_3$ be independent, nonnegative random variables defined in \eqref{Generalized Cauchy PDF}, and let $Z_{3,a}=X_1/(X_2X_3)^{2/a}.$ Then, Theorem \ref{Joint Density Generalized k Cauchy RVs Thm} implies
\[f_{Z_{3,a}}(z)= \frac{4a\sin^{3}(\frac{\pi}{a})}{\pi^3} \int_{0}^{\infty} \int_{0}^{\infty} \frac{x_2x_3}{(z^ax^2_2x^2_3+1)(x^2_2+1)(x^2_3+1)} \ dx_2 \ dx_3. \]

We use \eqref{Generalized PF} twice to see
\begin{align*}
f_{Z_{3,a}}(z) & = \frac{2a\sin^{3}(\frac{\pi}{a})}{\pi^3} \int_{0}^{\infty} \frac{x_3\ln(z^ax_3^2)}{(x_3^2+1)(z^ax_3^2-1)} \ dx_3 \\
& = \frac{2a^2\sin^{3}(\frac{\pi}{a})}{\pi^3} \int_{0}^{\infty} \frac{x_3\ln(z)}{(x_3^2+1)(z^ax_3^2-1)} \ dx_3 + \frac{4a\sin^{3}(\frac{\pi}{a})}{\pi^3} \int_{0}^{\infty} \frac{x_3\ln(x_3)}{(x_3^2+1)(z^ax_3^2-1)} \ dx_3 \\
& = \frac{a^3 \sin^{3}(\frac{\pi}{a})}{\pi^3} \frac{\ln^2(z)}{z^a+1} + \frac{4a\sin^{3}(\frac{\pi}{a})}{\pi^3} \int_{0}^{\infty} \frac{x_3\ln(x_3)}{(x_3^2+1)(z^ax_3^2-1)} \ dx_3 
\end{align*} 
Hence, using \eqref{Probability Axiom 1} we find that
\begin{equation}\label{J_3,a relation}
1  =\frac{a^3 \sin^{3}(\frac{\pi}{a})}{\pi^3} J_{3,a} + \frac{4a\sin^{3}(\frac{\pi}{a})}{\pi^3} \int_{0}^{\infty} \int_{0}^{\infty} \frac{x_3\ln(x_3)}{(x_3^2+1)(z^ax_3^2-1)} \ dx_3 \ dz.  
\end{equation}
We focus on evaluating the second term of \eqref{J_3,a relation}. Reversing the order of integration, and  making the substitution $z=vx_3^{-2/a},$ we see it is of the form in \eqref{Generalized Cauchy PV}. Thus, the second term is
\begin{equation} \label{2nd term}
-\frac{4\sin^{3}(\frac{\pi}{a}) \cot \left(\frac{\pi}{a} \right)}{\pi^2} \int_{0}^{\infty} \frac{x_3^{1-\frac{2}{a}}\ln(x_3)}{x_3^2+1} \ dx_3.
\end{equation}
Making the substitution, $x_3=u^\frac{a}{2a-2},$ we have \eqref{2nd term} is equal to
\begin{align*}
-\frac{a^2\sin^{3}(\frac{\pi}{a}) \cot \left(\frac{\pi}{a} \right)}{(a-1)^2\pi^2} \int_{0}^{\infty} \frac{\ln(u)}{u^{\frac{a}{a-1}} + 1} \ du
& = -\frac{a^2\sin^{3}(\frac{\pi}{a}) \cot \left(\frac{\pi}{a} \right)}{(a-1)^2\pi^2} \int_{0}^{\infty} \left(\frac{\ln(u)}{u^{\frac{a}{a-1}} + 1} -  \frac{2\ln(u)}{u^{\frac{2a}{a-1}} + 1} \right) \ du  \\
& = -\frac{a^2\sin^{3}(\frac{\pi}{a}) \cot \left(\frac{\pi}{a} \right)}{(a-1)^2\pi^2} \left(J_{2,\frac{a}{a-1}}-2J_{2,\frac{a}{a-1}} \right) \\
& = -\cos^2 \left(\frac{\pi}{a}\right).
\end{align*}
Now, we rearrange terms in \eqref{J_3,a relation} and use the fact $\cot(\theta)=\cos(\theta)/\sin(\theta)$ to see
\[ J_{3,a} = \frac{\pi^3}{a^3} \csc^{3} \left(\frac{\pi}{a}\right)+ \frac{\pi^3}{a^3} \cot^2 \left(\frac{\pi}{a}\right)\csc \left(\frac{\pi}{a}\right).\]
Thus, by Theorem \ref{S(k,a) J(k,a) Relation Theorem}, we have 
\[S(3,a) = \sum_{n=-\infty}^{\infty} \frac{(-1)^n}{(an+1)^3} 
= \frac{J_{3,a}}{(3-1)!} 
= \frac{\pi^3}{2a^3} \csc^{3} \left(\frac{\pi}{a}\right)+ \frac{\pi^3}{2a^3} \cot^{2} \left(\frac{\pi}{a}\right)\csc \left(\frac{\pi}{a}\right).\]
\end{proof}

\section{Uniform Random Variables}
We recall the Basel Problem solution by Beukers, Calabi and Kolk \cite{BCK} (see also \cite{NE,DK,silagadze2010}), whose generalization leads in to our analysis of the probability involving cyclically pairwise consecutive sums of $k$ independent uniform random variables on $(0,1).$    

\subsection{Beukers's, Calabi's, and Kolk's Basel Problem Solution}
We start with the double integral $D_2$ encountered in the introduction, which we modify as
\[D_2=\int_{0}^{1} \int_{0}^{1} \frac{1}{1-x^2_1x_2^2} \ dx_1 \ dx_2. \]
Rewriting the integrand as a geometric series, 
\begin{equation} \label{Beuker GS}
\frac{1}{1-x_1^2x_2^2} = \sum_{n=0}^{\infty}(x_1x_2)^{2n},
\end{equation}
and exchanging summation and integration, we see $D_2=S(2).$ 

Now, we modify Calabi's change of variables from \eqref{Calabi COV},
\[x_1=\frac{\sin\left(\frac{\pi}{2} u_1\right)}{\cos\left(\frac{\pi}{2} u_2\right)}, \quad x_2=\frac{\sin\left(\frac{\pi}{2} u_2\right)}{\cos\left(\frac{\pi}{2} u_1\right)}, \]
with Jacobian Determinant
\[\left | \frac{\partial(x_1,x_2)}{\partial (u_1,u_2)} \right | = \left (\frac{\pi}{2} \right)^2 \left( 1-x^2_1x^2_2 \right), \]
that transforms the integrand in $D_2$ into $\left(\pi/2 \right)^2.$ The resulting region of integration is the open isosceles right triangle with the shorter legs each having length $1.$ The area of this triangle is $1/2,$ so we recover $S(2)=\pi^2/8$ and $\zeta(2)=\pi^2/6.$

\subsection{From Integral to Polytope}

We now generalize the solution by Beukers, Calabi, and Kolk. We consider the multiple integral
\begin{equation}\label{Multi Integral 2}
\int_{0}^{1}\dots\int_{0}^{1} \frac{1}{1-(-1)^k\prod_{i=1}^{k} x^2_i} \ dx_1 \dots dx_k. 
\end{equation}
As before, rewriting the integrand of \eqref{Multi Integral 2} as a geometric series
\begin{equation} \label{Beuker Generalized GS}
\frac{1}{1-(-1)^k\prod_{i=1}^{k} x^2_i} = \sum_{n=0}^{\infty}(-1)^{nk} (x_1 \dots x_k)^{2n},
\end{equation}
exchanging summation and integration, and finally integrating term by term $k$ successive times yields \eqref{Multi Integral 2} is equal to $S(k).$ 

Then, we use the change of variables (modifying those in \cite{BCK,NE,silagadze2010})
\begin{equation}\label{COV}
x_i = \frac{\sin \left(\frac{\pi}{2}u_i \right)}{\cos \left( \frac{\pi}{2} u_{i+1} \right)}, \quad 1 \leq i \leq k,
\end{equation}
where the $u_i$ are indexed cyclically ${\rm mod\,} k,$ i.e. $u_{k+1}:=u_1.$ The Jacobian Determinant of \eqref{COV} is
\begin{equation}\label{Jacobian}
\left |\frac{\partial(x_1,\dots , x_k)}{\partial(u_1, \dots , u_k)} \right| = \frac{\pi^k}{2^k} \left(1-(-1)^k\prod_{i=1}^{k} x^2_i \right),
\end{equation}
and \eqref{COV} diffeomorphically maps the open cube $(0,1)^k$ into the open polytope
\[\Delta^k= \lbrace (u_1, \dots, u_k) \in \mathbb{R}^k : u_{i}+u_{i+1}<1, \ u_i>0, 1 \leq i \leq k \rbrace.\]
The proof of \eqref{Jacobian} is done both in \cite{BCK} and \cite{NE}. As a result,
\begin{equation}\label{S(k) Polytope Relation}
S(k)=\frac{\pi^k}{2^k}\text{Vol}(\Delta^k).
\end{equation}

Let $U_1, \dots , U_k$ be $k$ independent uniform random variables on $(0,1).$ Then, we see 
\begin{equation}\label{Volume Polytope}
\text{Vol}(\Delta^k)= \text{Pr}\left(U_1+U_2<1, \dots, U_{k-1} +U_{k}<1,U_{k}+U_{1}<1 \right).
\end{equation}
In the next section, we prove the formula
\[\textup{Vol}(\Delta^k)=\left( \frac{1}{2} \right)^k + \left(\frac{1}{2}\right)^k \sum_{n=1}^{\left \lfloor \frac{k}{2} \right \rfloor} \ \sum_{\substack{(r_1,\dots ,r_n) \in [k]^n :   
\\   \ |r_p-r_q| \notin \lbrace 0,1,k-1 \rbrace \\ p,q \in [n]}} \ \prod_{i=1}^{n} \frac{1}{i+\sum_{j=1}^{i} \alpha_{j} },\]
in which $[k]:= \lbrace 1, \dots k\rbrace,$ $\delta(a,b)$ is the Kronecker Delta Function
\[\delta(a,b) = \begin{cases}
1 & a=b \\
0 & \text{else}
\end{cases},\]
and $\alpha_j$ is defined
\[\alpha_j =2-\delta(k,2)-\sum_{m=1}^{j-1}  \delta(|r_m-r_j|, 2) + \delta(|r_m-r_j|, k-2).\]

\subsection{Polytope and Combinatorics}
For the sake of brevity, for $n \in \mathbb{N},$ we define 
\[[n] :=\lbrace 1, \dots , n \rbrace.\] 
We consider the probability
\begin{equation} \label{Polytope Probability}
\text{Pr}\left(U_1+U_2<1, \dots, U_{k-1} +U_{k}<1,U_{k}+U_{1}<1 \right),
\end{equation}
where $U_i$ is uniform on $(0,1)$ for each $i \in [k].$ Note that $f_{U_i}(u_i)=1$ for each $i \in [k]$ by \eqref{Uniform RV}. If $U_{r_1}, \dots, U_{r_n}$ are $n$ distinct random variables, we facilitate indexing by using cyclical convention: $U_{r_{n} + 1}:=U_{r_1}.$ We compute \eqref{Polytope Probability} by combining the probabilities of two cases: when all $U_i<1/2$ for each $i \in [k]$, and when $U_j \geq 1/2$ for some $j \in [k].$ The following theorem discusses the first case.

\begin{theorem} \label{3.3.1}
If $U_i<1/2$ for all $i \in [k],$ then \eqref{Polytope Probability} is equal to $\left(1/2\right)^k.$
\end{theorem}
\begin{proof}
For each $i \in [k],$ we have 
\[U_{i}+U_{i+1}< \frac{1}{2}+\frac{1}{2}=1.\]
Thus, \eqref{Polytope Probability} is the integral
\[\int_{0}^{\frac{1}{2}} \dots \int_{0}^{\frac{1}{2}}  \ du_{1} \dots \ du_{k}=\left(\frac{1}{2}\right)^k.\]
\end{proof}

\begin{remark*}
This probability geometrically represents the volume of the cube $\left(0,1/2\right)^k.$
\end{remark*}
The second case, namely when $U_j \geq 1/2,$ for some $j \in [k],$ is nontrivial. We discuss which random variables can be simultaneously at least $1/2.$ 

\begin{theorem}\label{3.3.2}
Suppose $U_{i}+U_{i+1}<1$ for all $i \in [k].$ Let there be distinct indices $r_1, \dots ,r_n \in [k]$ such that $U_{r_1}, \dots, U_{r_n} \geq 1/2.$ Then, we have that $r_1, \dots, r_n$ are cyclically pairwise nonconsecutive. That is, for all distinct $p,q \in [n],$ we have 
$|r_p-r_q| \notin \lbrace 1,k-1 \rbrace.$ 
\end{theorem}

\begin{proof}
Suppose on the contrary we had two distinct numbers $p,q \in [n]$ with $|r_p-r_q| \in \lbrace 1, k-1 \rbrace$ such that $U_{r_p},U_{r_q} \geq 1/2.$ 

If $|r_p-r_q|=1,$ we must have $r_p$ and $r_q$ consecutive to one another. Hence $U_{r_p}+U_{r_q}<1,$ which is a contradiction.

If $|r_p-r_q|=k-1,$ we must have $r_p=1$ and $r_q=k,$ or vice versa. But in either case, we see $U_{1}+U_{k}<1,$ which is a contradiction.
\end{proof}

Now, we find the maximal $n$ that would satisfy the previous theorem. 

\begin{theorem} \label{3.3.3}
Reconsider the hypotheses from Theorem \ref{3.3.2}. Then we have $n \leq \lfloor k/2 \rfloor.$
\end{theorem}

\begin{proof}
Suppose $k$ is even and $n > \lfloor k/2 \rfloor=k/2.$ Let $\sigma$ be an increasing permutation map making $\sigma(r_1) < \dots < \sigma(r_n).$ Note that $\sigma(r_1), \dots , \sigma(r_n)$ are cyclically pairwise nonconsecutive. Hence, for all $j \in [n-1],$ there is an integer $x_j \in [k]$ such that $\sigma(r_j)<x_j<\sigma(r_{j+1}).$ There is also an $x_n \in [k]$ such that either $x_n>\sigma(r_{n})$ or $x_n<\sigma(r_{1}).$ In all we have enumerated at least $2n$ distinct integers in the set $[k].$ Since $n>k/2,$ we see $2n>k.$ Hence, we have enumerated more than $k$ integers in $[k],$ which is impossible. 

Now, suppose $k$ is odd and $n > \lfloor k/2 \rfloor=(k-1)/2.$ Construct an increasing permutation map $\sigma$ as before. This would result in us enumerating more than $k-1$ distinct integers in the set $[k].$ Thus, we must have $2n=k,$ so $2|k.$ This is a contradiction to $k$ being odd.    
\end{proof}

\begin{theorem}\label{3.3.4}
Let $y_1, \dots ,y_n \in \mathbb{R}$ such that $1/2 \leq y_n \leq ... \leq y_1 < 1,$ and let $\beta_1, \dots ,\beta_n$ be nonnegative integers. Define the function 
\begin{equation}\label{polynomial}
\psi(y_1,\dots , y_n) = \prod_{i=1}^{n}(1-y_i)^{\beta_i}. 
\end{equation}
Then the following integral formula holds 
\begin{equation} \label{integral formula}
\int_{\frac{1}{2}}^{1} \int_{\frac{1}{2}}^{y_1} \dots \int_{\frac{1}{2}}^{y_{n-1}} \psi(y_1,\dots , y_n) \ dy_n \dots dy_1= \left( \frac{1}{2}\right) ^{n+\sum_{j=1}^{n} \beta_j} \prod_{i=1}^{n} \frac{1}{i+\sum_{j=1}^{i} \beta_{j}}.
\end{equation}
\end{theorem}

\begin{proof}
We first make the change of variables $y_j=1-t_j, j \in [n],$ which transforms the left hand side of \eqref{integral formula} into 
\begin{align}
 \int_{0}^{\frac{1}{2}} \int_{t_1}^{\frac{1}{2}} \dots \int_{t_{n-1}}^{\frac{1}{2}} \prod_{i=1}^{n} t_i^{\beta_i} \ dt_n \dots \ dt_1 & =  \int_{0}^{\frac{1}{2}} \int_{0}^{t_n} \dots \int_{0}^{t_2} \prod_{i=1}^{n} t_i^{\beta_i} \ dt_1 \dots \ dt_n \label{Reverse Order} \\
 & = \left( \frac{1}{2}\right) ^{n+\sum_{j=1}^{n} \beta_j} \frac{1}{1+\beta_1} \frac{1}{2+\beta_1+\beta_2} \dots \frac{1}{n+\sum_{j=1}^{n}\beta_j} \label{Induction Argument} \\ 
 & = \left( \frac{1}{2}\right) ^{n+\sum_{j=1}^{n} \beta_j} \prod_{i=1}^{n} \frac{1}{i+\sum_{j=1}^{i} \beta_{j}} \nonumber.
\end{align}
The first equality of \eqref{Reverse Order} follows from reversing the order of integration, and \eqref{Induction Argument} follows upon induction on $n.$  

\end{proof}

We discuss the integral bounds for the multiple integral corresponding to the probability in \eqref{Polytope Probability}.

\begin{theorem}\label{3.3.5}
Reconsider the hypotheses from Theorem \ref{3.3.2}, but this time, suppose further that $1/2 \leq U_{r_n} \leq \dots \leq U_{r_1} <1.$ For each $j \in [n],$ let $\alpha_j$ be the number of $x \in \lbrace r_j \pm 1 \rbrace$ such that $U_x<1-U_{r_j}\leq 1-U_{r_l}$ holds for all $l \in [n] \setminus \lbrace j \rbrace.$ Then
\begin{equation}\label{exponents}
\alpha_j= 2-\delta(k,2)-\sum_{m=1}^{j-1} \delta(|r_m-r_j|, 2) + \delta(|r_m-r_j|, k-2).  
\end{equation} 
\end{theorem}

\begin{proof}
Suppose $k=2.$ Then by Theorem \ref{3.3.3}, exactly one of $U_1,U_2$ must be at least $1/2,$ so $\alpha_1=1$ in each case. This agrees with the right hand side of \eqref{exponents}, since the summation term vanishes.

Let $k>2.$ For $j=1,$ since $U_{r_1}$ is the greatest random variable, we have $1-U_{r_1} \leq 1-U_{r_l}$ for each $l \in [n] \setminus \lbrace 1 \rbrace.$ So $U_{r_1-1}, U_{r_1+1} < 1-U_{r_1},$ which means $\alpha_1=2.$ 
For $j>1,$ we observe if $\lbrace r_j \pm 1 \rbrace \cap \lbrace r_m \pm 1 \rbrace$ is nonempty for some $m\in [j-1],$ then $|r_m-r_j| \in \lbrace 2, k-2 \rbrace.$ If $x \in \lbrace r_j \pm 1 \rbrace \cap \lbrace r_m \pm 1 \rbrace,$ then $U_x< 1-U_{r_m} \leq 1-U_{r_j}$ since $U_{r_j}\leq U_{r_m}.$ We count the number of $x \in \lbrace r_j \pm 1 \rbrace \cap \lbrace r_m \pm 1 \rbrace$ for some $m \in [j-1]$ by
\[\sum_{m=1}^{j-1} \delta(|r_m-r_j|, 2) + \delta(|r_m-r_j|, k-2).\] Hence, we must have 
\[\alpha_j=2-\sum_{m=1}^{j-1} \delta(|r_m-r_j|, 2) + \delta(|r_m-r_j|, k-2).\] This agrees with the right hand side of \eqref{exponents} since $\delta(k,2)=0$ for $k > 2.$
\end{proof}

\begin{theorem}\label{3.3.6}
Reconsider the hypotheses from Theorem \ref{3.3.5}. Then the probability expressed by \eqref{Polytope Probability} is
\begin{equation} \label{big probability}
\left( \frac{1}{2}\right) ^{k} \prod_{i=1}^{n} \frac{1}{i+\sum_{j=1}^{i}\alpha_{j}},
\end{equation}
with $\alpha_j$ defined as in \eqref{exponents}.
\end{theorem}

\begin{proof}

We have from our hypothesis the bounds
\begin{equation}\label{Bounds}
\begin{cases}
\frac{1}{2} \leq U_{r_j} < 1  & j=1 \\
\frac{1}{2} \leq U_{r_j} \leq U_{r_{j-1}} & j \in [n] \setminus \lbrace 1 \rbrace. 
\end{cases}
\end{equation}
For each $j \in [n],$ there are $\alpha_j$ bounds of the form $0<U_{x}<1-U_{r_j},$ with $\alpha_j$ constructed in \eqref{exponents}. As a result, there are $k-n-\sum_{j=1}^{n} \alpha_j$ remaining bounds of the form $0<U_{x}<1/2.$

We now set up the integral corresponding to the probability \eqref{Polytope Probability}. We integrate first with respect to all $u_x$ with $0<U_{x}<1/2.$ Then, we integrate with respect to all $u_x$ in which $0<U_{x}<1-U_{r_j}$ for some $j \in [n].$ This results in \eqref{Polytope Probability} simplifying down to
\begin{equation} \label{Prob U_x >1/2}
\int_{\frac{1}{2}}^{1} \int_{\frac{1}{2}}^{u_{r_1}} \dots \int_{\frac{1}{2}}^{u_{r_{n-1}}} \left(\frac{1}{2}\right)^{k-n-\sum_{j=1}^{n} \alpha_j}\phi(u_{r_1}, \dots, u_{r_n}) \ du_{r_n} \dots \ du_{r_1},
\end{equation}
where 
\[\phi(u_{r_1}, \dots, u_{r_n})= \prod_{i=1}^{n}(1-u_{r_i})^{\alpha_i}.\]
Recognizing $\phi$ is a function of the form in \eqref{polynomial}, we use Theorem \ref{3.3.4} to see \eqref{Prob U_x >1/2} is 
\[\left( \frac{1}{2}\right) ^{k} \prod_{i=1}^{n} \frac{1}{i+\sum_{j=1}^{i} \alpha_j }. \]

\end{proof}

We now can prove the capstone results of our combinatorial analysis.

\begin{theorem} \label{3.3.7}
We have 
\begin{equation}\label{Volume Closed Form}
\textup{Vol}(\Delta^k)=\left( \frac{1}{2} \right)^k + \left(\frac{1}{2}\right)^k \sum_{n=1}^{\left \lfloor \frac{k}{2} \right \rfloor} \ \sum_{\substack{(r_1,\dots ,r_n) \in [k]^n :   
\\   \ |r_p-r_q| \notin \lbrace 0,1,k-1 \rbrace \\ p,q \in [n]}} \ \prod_{i=1}^{n} \frac{1}{i+\sum_{j=1}^{i} \alpha_{j} },
\end{equation}
\begin{equation} \label{S(k) Closed Form}
S(k) =\left( \frac{\pi}{4} \right)^k + \left(\frac{\pi}{4}\right)^k \sum_{n=1}^{\left \lfloor \frac{k}{2} \right \rfloor} \ \sum_{\substack{(r_1,\dots ,r_n) \in [k]^n :   
\\   \ |r_p-r_q| \notin \lbrace 0,1,k-1 \rbrace \\ p,q \in [n]}} \ \prod_{i=1}^{n} \frac{1}{i+\sum_{j=1}^{i} \alpha_{j} }, 
\end{equation}
\begin{equation}\label{zeta(2k) Closed Form}
\zeta(2k) =\frac{\pi^{2k}}{2^{2k}-1} + \frac{\pi^{2k}}{2^{2k}-1} \sum_{n=1}^{k} \ \sum_{\substack{(r_1,\dots ,r_n) \in [2k]^n :   
\\   \ |r_p-r_q| \notin \lbrace 0,1,2k-1 \rbrace \\ p,q \in [n]}} \ \prod_{i=1}^{n} \frac{1}{i+\sum_{j=1}^{i} \alpha_{j} },
\end{equation}
with $\alpha_j$ defined in \eqref{exponents}. In particular the inner sums of \eqref{Volume Closed Form}, \eqref{S(k) Closed Form}, and \eqref{zeta(2k) Closed Form} are taken over all tuples in $[k]^n$ having cyclically pairwise nonconsecutive entries. 
\end{theorem}
\begin{proof}
The first term in \eqref{Volume Closed Form} follows from Theorem \ref{3.3.1}. Next, we consider all possible tuples of length $n$ with entries in $[k]$ satisfying the hypotheses for Theorem \ref{3.3.2}. We use Theorem \ref{3.3.6} to compute corresponding probability \eqref{Polytope Probability} for each of those tuples. We sum of all these probabilities together to get the second term. We see \eqref{S(k) Closed Form} immediately follows from \eqref{S(k) Polytope Relation}, and \eqref{zeta(2k) Closed Form} follows directly from \eqref{Zeta(2k) in terms of S(2k)} and the fact $\left \lfloor 2k/2 \right \rfloor=k.$
\end{proof}

\begin{remark*}
When $k=1,$ Theorem \ref{3.3.3} implies there are no random variables that can be at least $1/2,$ which then results in the second terms of \eqref{Volume Closed Form} and \eqref{S(k) Closed Form} vanishing. However, the first terms stay put due to Theorem \ref{3.3.1}.
\end{remark*}

\section{Cauchy Random Variables Revisited}
We recall the Basel Problem solution by Zagier and Kontsevich \cite[p. 9]{ZK} and then show how its generalization leads us a probability involving the  cyclically pairwise consecutive products $k$ independent, nonnegative Cauchy random variables. 
\subsection{Zagier's and Kontsevich's Basel Problem Solution}
We start with the double integral $D_3,$ which we modify as
\begin{equation} \label{Double Integral 3 Modified}
\frac{4}{\pi^2} \int_{0}^{1} \int_{0}^{1} \frac{1}{4\sqrt{x_1x_2}(1-x_1x_2)} \ dx_1 \ dx_2.
\end{equation} 
From a substitution $x=\sqrt{u},y=\sqrt{v}$ on $D_2,$ we see that \eqref{Double Integral 3 Modified} is $4S(2)/{\pi^2}.$ Using the change of variables from \eqref{Zagier COV},
\[x_1=  \frac{\xi^2_1(\xi^2_2+1)}{\xi^2_1+1},  \quad x_2=  \frac{\xi^2_2(\xi^2_1+1)}{\xi^2_2+1},\]
with Jacobian Determinant 
\[\left | \frac{\partial(x_1,x_2)}{\partial(\xi_1,\xi_2)} \right | = \frac{4\sqrt{x_1x_2}(1-x_1x_2)}{(\xi^2_1+1)(\xi^2_2+1)},\]
\eqref{Double Integral 3 Modified} becomes
\begin{align}
\frac{4}{\pi^2}\int_{0}^{\infty} \int_{0}^{\frac{1}{\xi_1}} \frac{1}{(\xi^2_1+1)(\xi^2_2+1)} \ d\xi_2 \ d\xi_1 & = \frac{2}{\pi^2} \int_{0}^{\infty} \int_{0}^{\infty} \frac{1}{(\xi^2_1+1)(\xi^2_2+1)} \ d\xi_2 \ d\xi_1  \label{Involute} \\
& =  \frac{1}{2} \nonumber ,
\end{align}
in which the equality \eqref{Involute} follows from a symmetry argument. Thus, we see
\[\frac{4}{\pi^2}S(2)=\frac{1}{2},\] which implies $S(2)=\pi^2/8.$

\subsection{From Integral to Hyperbolic Polytope}
The analogue of \eqref{Double Integral 3 Modified} is
\begin{equation} \label{Double Integral 3 General Modified}
\frac{2^k}{\pi^k} \int_{0}^{1} \int_{0}^{1} \frac{1}{2^k \sqrt{x_1 \dots x_k}(1-(-1)^{k}x_1 \dots x_k)} \ dx_1 \dots \ dx_k,
\end{equation} 
which is $(2/\pi)^kS(k)$ using a change of variables $x_i=\sqrt{u_i}, 1 \leq i \leq k$ on \eqref{Multi Integral 2}. The change of variables generalizing \eqref{Zagier COV} is
\begin{equation} \label{Generalized Zagier COV}
x_i= \frac{\xi^2_i(\xi^2_{i+1}+1)}{\xi^2_i+1}, \quad 1 \leq i \leq k,
\end{equation}
with cyclical indexing mod $k:$ $\xi_{k+1}:=\xi_{1}.$ The change of variables in \eqref{Generalized Zagier COV} has Jacobian Determinant
\begin{align} \label{Generalized Zagier Determinant}
\left |\frac{\partial(x_1, \dots, x_k)}{\partial(\xi_1, \dots, \xi_k)} \right |= \frac{2^k \sqrt{x_1 \dots x_k}(1-(-1)^{k}x_1 \dots x_k)}{(\xi_{1}^2+1) \dots (\xi_{k}^2+1)},
\end{align}
and diffeomorphically maps $(0,1)^k$ to the hyperbolic polytope defined in \eqref{Hypertope}, namely
\begin{align*}
\mathbb{H}^k= \lbrace (\xi_1, \dots, \xi_k) \in \mathbb{R}^k : \xi_i\xi_{i+1}<1, \ \xi_i>0, 1 \leq i \leq k \rbrace.
\end{align*}
The proof of the Jacobian Determinant is remarkably similar in structure to the proof of \eqref{Jacobian}.

If $\Xi_1, \dots, \Xi_k$ are independent, nonnegative Cauchy random variables, then 
\begin{align} \label{S(k) Hypertope Relation}
\frac{2^k}{\pi^k}S(k)=\text{Pr}(\Xi_1\Xi_2<1, \dots, \Xi_{k-1}\Xi_{k}<1, \Xi_{k}\Xi_{1}<1).
\end{align} 
We already know  the left hand side of \eqref{S(k) Hypertope Relation} is $\text{Vol}(\Delta^k),$ which we found by combinatorially analyzing the probability from \eqref{Polytope Probability}. We show the connections between the analyses of \eqref{Polytope Probability} and 
\begin{align}\label{Hypertope Probability}
\text{Pr}(\Xi_1\Xi_2<1, \dots, \Xi_{k-1}\Xi_{k}<1, \Xi_{k}\Xi_{1}<1),
\end{align}
which both lead to the formulas in \eqref{Volume Closed Form}, \eqref{S(k) Closed Form}, and \eqref{zeta(2k) Closed Form}.

\subsection{Hyperbolic Polytope and Combinatorics}
We state the analogues of Theorems \ref{3.3.1}-\ref{3.3.6}, and prove only those directly dealing with \eqref{Hypertope Probability}. The omitted proofs are similar to those in the previous section. 

\begin{theorem} \label{4.3.1}
If $\Xi_i<1$ for all $i \in [k],$ then \eqref{Hypertope Probability} is equal to $\left(1/2\right)^k.$
\end{theorem}

\begin{proof}
The integral corresponding to \eqref{Hypertope Probability} is
\begin{align*}
\frac{2^k}{\pi^k} \int_{0}^{1} \dots \int_{0}^{1} f_{\Xi_1}(\xi_1) \dots f_{\Xi_k}(\xi_k) \ d\xi_k \dots d\xi_1 & = \frac{2^k}{\pi^k} \left(\int_{0}^{1} \frac{1}{\xi_1^2+1} d\xi_1 \right)^k  \\
& = \frac{2^k}{\pi^k} \frac{\pi^k}{4^k} = \left(\frac{1}{2}\right)^k.
\end{align*}
\end{proof}

\begin{theorem}\label{4.3.2}
Suppose $\Xi_{i}\Xi_{i+1}<1$ for all $i \in [k].$ Let there be distinct indices $r_1, \dots ,r_n \in [k]$ such that $\Xi_{r_1}, \dots \Xi_{r_n}, \geq 1$ Then for all distinct $p,q \in [n],$ we have 
$|r_p-r_q| \notin \lbrace 1,k-1 \rbrace.$ 
\end{theorem}
\begin{proof}
Similar to the proof of Theorem \ref{3.3.2}.
\end{proof}

\begin{theorem} \label{4.3.3}
Reconsider the hypotheses from Theorem \ref{4.3.2}. Then we have $n \leq \lfloor k/2 \rfloor.$
\end{theorem}
\begin{proof}
Similar to the proof of Theorem \ref{3.3.3}.
\end{proof}

\begin{theorem}\label{4.3.4}
Given $y_1, \dots ,y_n \in \mathbb{R}$ with $1 \leq y_n \leq \dots \leq y_1,$ and nonnegative integers $\beta_1, \dots, \beta_n,$ define 
\begin{align} \label{hyperpolynomial}
\phi(y_1, \dots, y_n)=\prod_{i=1}^{n}  \frac{\left(\cot^{-1}(y_i) \right)^{\beta_i}}{y^2_i+1}.
\end{align}
Then the following integral formula holds
\begin{equation} \label{hypertope integral formula}
\int_{1}^{\infty} \int_{1}^{y_1} \dots \int_{1}^{y_{n-1}} \phi(y_1, \dots, y_n) \ dy_1 \dots \ dy_n  = \left( \frac{\pi}{4}\right) ^{n+\sum_{j=1}^{n} \beta_j} \prod_{i=1}^{n} \frac{1}{i+\sum_{j=1}^{i} \beta_{j}}.
\end{equation}
\end{theorem}

\begin{proof}
We make the change of variables $y_i=1/z_i, i \in [n],$ to see the left hand side of \eqref{hypertope integral formula} becomes 
\begin{align*}
\int_{0}^{1} \int_{z_1}^{1} \dots \int_{z_{n-1}}^{1}\prod_{i=1}^{n}  \frac{\left(\tan^{-1}(z_i) \right)^{\beta_i}}{z^2_i+1} \ dz_n \dots \ dz_1 & = \int_{0}^{1} \int_{0}^{z_n} \dots \int_{0}^{z_2}\prod_{i=1}^{n} \frac{\left(\tan^{-1}(z_i) \right)^{\beta_i}}{z^2_i+1} \ dz_1 \dots \ dz_n \\
& = \left( \frac{\pi}{4}\right) ^{n+\sum_{j=1}^{n} \beta_j} \prod_{i=1}^{n} \frac{1}{i+\sum_{j=1}^{i} \beta_{j}},
\end{align*}
which is seen upon induction on $n.$ 
\end{proof}
\begin{theorem}\label{4.3.5}
Reconsider the hypotheses from Theorem \ref{4.3.2}, but this time, let us have $1 \leq \Xi_{r_n} \leq \dots \leq \Xi_{r_1}.$ For each $j \in [n],$ let $\alpha_j$ be the number of $x \in \lbrace r_j \pm 1 \rbrace$ such that $\Xi_x<1/{\Xi_{r_j}} \leq 1/{\Xi_{r_l}}$ holds for all $l \in [n] \setminus \lbrace j \rbrace.$ Then we see $\alpha_j$ is the same as it is in \eqref{exponents}, namely
\[\alpha_j= 2-\delta(k,2)-\sum_{m=1}^{j-1} \delta(|r_m-r_j|, 2) + \delta(|r_m-r_j|, k-2).  \]
\end{theorem}

\begin{proof}
Similar to the proof of Theorem \ref{3.3.5}.
\end{proof}

\begin{theorem}\label{4.3.6}
Reconsider the hypotheses from Theorem \ref{4.3.5}. Then we have that \eqref{Hypertope Probability} is equal to the right hand side of \eqref{big probability}, namely,
\[\left( \frac{1}{2}\right) ^{k} \prod_{i=1}^{n} \frac{1}{i+\sum_{j=1}^{i}\alpha_{j}},\]
with $\alpha_j$ defined as in \eqref{exponents}.
\end{theorem}

\begin{proof}
We have from our hypothesis the bounds
\begin{equation}\label{Hypertope Bounds}
\begin{cases}
1 \leq \Xi_{r_j} < \infty  & j=1 \\
1 \leq \Xi_{r_j} \leq \dfrac{1}{\Xi_{r_{j-1}}} & j \in [n] \setminus \lbrace 1 \rbrace. 
\end{cases}
\end{equation}
For each $j \in [n],$ there are $\alpha_j$  bounds of the form $0<\Xi_{x}<1/{\Xi_{r_j}},$ with $\alpha_j$ constructed in \eqref{exponents}, and $k-n-\sum_{j=1}^{n} \alpha_j$ bounds of the form $0<\Xi_{x}<1.$

We now set up the integral for \eqref{Hypertope Probability}, whose integrand is $f_{\Xi_1}(\xi_1) \dots f_{\Xi_k}(\xi_k).$ We integrate first with respect to all $\xi_x$ with bounds of the form $0<\Xi_{x}<1.$ Then, we integrate with respect to all $\xi_x$ with $0<\Xi_{x}<1/{\Xi_{r_j}}$ for some $j \in [n].$ Lastly, we integrate with respect to $\xi_{r_j}, \dots, \xi_{r_1}$ in that order with the bounds of the form in \eqref{Hypertope Bounds}.

Integrating with respect to the first two groups of variables, \eqref{Hypertope Probability} becomes  
\begin{equation} \label{Prob Xi_x >1}
\frac{2^k}{\pi^k} \int_{1}^{\infty} \int_{1}^{\frac{1}{\xi_{r_1}}} \dots \int_{1}^{\frac{1}{\xi_{r_{n-1}}}} \left(\frac{\pi}{4}\right)^{k-n-\sum_{j=1}^{n} \alpha_j}\phi(\xi_{r_1}, \dots, \xi_{r_n}) \ d\xi_{r_n} \dots \ d\xi_{r_1},
\end{equation}
where 
\[\phi(\xi_{r_1}, \dots, \xi_{r_n})= \prod_{i=1}^{n}  \frac{\left(\cot^{-1}(\xi_i) \right)^{\alpha_i}}{\xi^2_i+1}.\]
Recognizing $\phi$ is a function of the form \eqref{hyperpolynomial}, we use Theorem \ref{4.3.4} to simplify \eqref{Prob Xi_x >1} to 
\[\left( \frac{1}{2}\right) ^{k} \prod_{i=1}^{n} \frac{1}{i+\sum_{j=1}^{i} \alpha_j }. \]

\end{proof}
Therefore,
\[S(k)=\frac{\pi^k}{2^k}\text{Vol}(\Delta^k)=\frac{\pi^k}{2^k} \int_{\mathbb{H}^k}f_{\Xi_1}(\xi_1) \dots f_{\Xi_k}(\xi_k) \ d\xi_1 \dots \ d\xi_k, \] and the capstone results from the previous section are immediate consequences. 

As promised, we have shown $S(k)$ and $\zeta(2k)$ are the consequences of two equal probabilities: the probability $k$ uniform random variables on $(0,1)$ have cyclically pairwise consecutive sums less than $1,$ and the probability $k$ independent, nonnegative Cauchy random variables have cyclically pairwise consecutive products less than $1.$

\end{document}